\documentclass[12pt,reqno]{amsart}
\usepackage{color}

\usepackage{amsmath , amssymb, latexsym }
\usepackage{graphics}
\usepackage{graphicx}
\usepackage{verbatim} 

\usepackage[mathscr]{eucal}

\newtheorem{thm}{Theorem}
\newtheorem{conj}{Conjecture}

\newtheorem{lemma}[thm]{Lemma}

\newtheorem{cor}[thm]{Corollary}
\newtheorem{claim}{Claim}

\theoremstyle{definition}
\theoremstyle{definition}
\theoremstyle{definition}
\theoremstyle{definition}

\newcommand\spn[1]{\mathrm{span}(#1)}

\def\HH{\mathcal{H}}

\def\K{\mathcal{K}}

\def\P{\mathcal{P}}

\def\N{\mathbb{N}}

\def\RR{\mathbb{R}}

\def\le{\leqslant}
\def\ge{\geqslant}

\def\->{\rightarrow}
\def\<{\langle}
\def\>{\rangle}

\def\x{\mathbf{x}}

\def\t{\mathbf{t}}

\def\0{\mathbf{0}}
\def\1{\mathbf{1}}

\begin{document}

\title{Linear algebra and bootstrap percolation}

\author{J\'ozsef Balogh}
\address{Department of Mathematics\\ University of Illinois\\ 1409 W. Green Street\\ Urbana, IL 61801\\ and\\ Department of Mathematics\\ University of California\\ San Diego, La Jolla, CA 92093}\email{jobal@math.uiuc.edu}

\author{B\'ela Bollob\'as}
\address{Trinity College\\ Cambridge CB2 1TQ\\ England\\ and \\ Department of Mathematical Sciences\\ The University of Memphis\\ Memphis, TN 38152, USA} \email{B.Bollobas@dpmms.cam.ac.uk}

\author{Robert Morris}
\address{IMPA, Estrada Dona Castorina 110, Jardim Bot\^anico, Rio de Janeiro, RJ, Brasil} \email{rob@impa.br}

\author{Oliver Riordan}
\address{Mathematical Institute, University of Oxford, 24--29 St Giles', Oxford OX1 3LB, UK}\email{riordan@maths.ox.ac.uk}

\thanks{Research supported in part by: (JB) NSF CAREER Grant DMS-0745185, UIUC Campus Research Board Grants 09072 and 11067, OTKA Grant K76099, and the TAMOP-4.2.1/B-09/1/KONV-2010-0005 project; (BB) NSF grants DMS-0906634, CNS-0721983 and CCF-0728928, ARO grant W911NF-06-1-0076, and TAMOP-4.2.2/08/1/2008-0008 program of the Hungarian Development Agency; (RM) CNPq bolsa de Produtividade em Pesquisa}

\begin{abstract}
In $\HH$-bootstrap percolation, a set $A \subset V(\HH)$ of initially
`infected' vertices spreads by infecting vertices which are the only
uninfected vertex in an edge of the hypergraph $\HH$. A particular
case of this is the $H$-bootstrap process, in which $\HH$ encodes
copies of $H$ in a graph $G$. We find the minimum size of a set $A$
that leads to complete infection when $G$ and $H$ are powers of complete
graphs and $\HH$ encodes induced copies of $H$ in $G$. The proof uses linear algebra, a
technique that is new in bootstrap percolation, although standard in
the study of weakly saturated graphs, which are equivalent to (edge)
$H$-bootstrap percolation on a complete graph.
\end{abstract}

\maketitle

\section{Introduction}\label{intro}

Given a hypergraph $\HH$, the \emph{$\HH$-bootstrap process} is defined as follows. Let $A \subset V(\HH)$ be a set of initially `infected' vertices, and, at each time step, infect a vertex $u$ if it lies in an edge of $\HH$ in which all vertices other than $u$ are already infected. To be precise, set $A_0 = A$, and, for each $t \ge 0$, set
\[
 A_{t+1} \, := \, A_t \,\cup\, \big\{ u \,:\, \exists \, S \in \HH \; \textup{ with }\; S \setminus A_t = \{u\} \big\}.
\]
Let $[A]_\HH = \bigcup_{t\ge0} A_t$, and say that $A$ \emph{percolates} (or $\HH$-percolates) if $[A]_\HH = V(\HH)$.

A large family of models of this type was introduced in~\cite{bb11}. Given graphs $G$ and $H$, we obtain the \emph{$H$-bootstrap process} on $G$ by setting $\HH = \{ V(H') \,:\, H' \subset G \textup{ and } H' \cong H \}$. (Sometimes it is more natural to consider 
only induced copies of $H$ in $G$.) The $\HH$- and $H$-bootstrap processes can be seen as special cases of the `cellular automata' introduced by von Neumann (see~\cite{vN}) after a suggestion of Ulam~\cite{Ulam}, and generalize several previously studied models. For example, if $G$ is a (finite) square grid and $H = C_4$  
then we obtain the so-called `Frob\"ose process' (see~\cite{Frob} or~\cite{GHM}).

A fundamental question about bootstrap-type models is the following: given a hypergraph $\HH$ (or a pair $(G,H)$), how large is the smallest percolating set in the $\HH$-bootstrap process? We define
$$m(\HH) \; := \; \min\big\{ |A| \,:\, A \subset V(\HH), \;[A]_\HH = V(\HH) \big\}.$$
Let $K_n^d$ denote the graph with vertex set $[n]^d = \{1,\ldots,n\}^d$ 
in which $uv$ is an edge if $u$ and $v$ differ in exactly one coordinate.
Given $1 \le r \le d$ and $2 \le t\le n$, let $\K(n,d,t,r)$ be the hypergraph
with vertex set $[n]^d$ in which the edges are all sets $S$ of the form $S=I_1\times I_2\times\cdots I_d$
where $r$ of the sets $I_j\subset [n]$ have size $t$ and the others are singletons.
Equivalently,
$$E\big( \K(n,d,t,r) \big) \,:=\, \big\{ S \subset [n]^d \,:\, K_n^d[S] \cong K_t^r \big\},$$
the collection of induced copies of $K_t^r$ in $K_n^d$.
Note that $K_2^r = Q_r$, the $r$-dimensional hypercube.

Our main aim is to determine $m\big( \K(n,d,t,r) \big)$ precisely for every $n\ge t\ge 2$ and $d \ge r \ge 1$. We shall also consider the grid $P_n^d$ with vertex set $[n]^d$, in which two vertices are adjacent if they differ by 1 in one coordinate, and agree in all others. (This graph is usually denoted $[n]^d$, but here this notation would cause confusion.) The corresponding hypergraph $\P(n,d,t,r)$ has as edges all sets
$S=I_1\times I_2\times\cdots I_d$ where $r$ of the $I_j$ are \emph{intervals} of size $t$, and the rest are singletons.
Note that while such sets $S$ induce copies of $P_t^r$ in $P_n^d$, they are not the only induced copies. (There can also be `bent' copies.)
Clearly $\P(n,d,t,r) \subset \K(n,d,t,r)$. As noted in Section~\ref{oliver}, below, the set of points 
\begin{equation}\label{Udef}
U \, = \, \big\{ (u_1,\ldots,u_d)\in [n]^d \,:\, \big| \big\{ i : u_i \ge t \big\} \big| \le r-1 \big\}
\end{equation}
percolates in $\P(n,d,t,r)$, and hence also in $\K(n,d,t,r)$. Our main theorem implies that the set $U$ is extremal in both hypergraphs.

\begin{thm}\label{mainthm}
For every $n \ge t \ge 2$ and $d \ge r \ge 1$,
$$m\big(\K(n,d,t,r) \big) \, = \, m\big(\P(n,d,t,r) \big) \, = \, \sum_{s=0}^{r-1} \binom{d}{s} (t-1)^{d-s} (n+1-t)^s.$$
\end{thm}

The special case $r = d$ of this result was proved over 25 years ago by Alon~\cite{alon} using techniques from exterior algebra. In this case $U$ is simply the set of points in which at least one coordinate is in the range $1$ up to $t - 1$.

\begin{cor}[{\cite[Theorem~3.4]{alon}}]\label{aloncor}
For every $n \ge t \ge 2$ and $d \ge 1$ we have
$$m\big(\K(n,d,t,d) \big) \, = \, m\big(\P(n,d,t,d) \big) \, = \, n^d - ( n + 1 - t )^d.$$
\end{cor}

We remark that Alon's Theorem was phrased in terms of edge percolation in complete multi-partite hypergraphs\footnote{Translating from edges in a hypergraph to vertices in the line graph maps his result to Corollary~\ref{aloncor}.}, and was somewhat more general than Corollary~\ref{aloncor}. Indeed, given a vector $\t = (t_1,\ldots,t_d) \in \N^d$, let $\K^*(n,d,\t)$ denote the hypergraph on $[n]^d$ whose edges are copies of $K_{t_1} \times \ldots \times K_{t_d}$ which have `length' $t_j$ in direction $j$ for each $1 \le j \le d$. That is, those copies which sit `as we expect' and with a prescribed orientation. In~\cite{alon}, Alon determined $m\big( \K^*(n,d,\t) \big)$ for every $n,d \in \N$ and $\t \in \N^d$; thus Theorem~\ref{mainthm} can be seen as the natural analogue of Alon's Theorem when we allow the copies of $K_t^r$ to be oriented in any direction. In Section~\ref{discuss} we present a result that generalizes both Alon's result and Theorem~\ref{mainthm}. 

Before turning to the proof of Theorem~\ref{mainthm}, let us give a little context. The first extremal result related to bootstrap percolation was proved by Bollob\'as~\cite{Bela68}, and phrased in the language of `weakly saturated graphs'. This is the natural edge version of the  $H$-bootstrap percolation we have just defined (infect an edge if it is the last uninfected edge of a copy of $H$), with $G$ complete. The main aim of \cite{Bela68} was to pose a conjecture concerning 
the extremal number when $H=K_k$ and $G=K_n$.
This conjecture was proved by Alon~\cite{alon}, Frankl~\cite{Fr} and Kalai~\cite{Kalai},
using linear algebraic methods.  

The $\HH$-bootstrap process is named after a closely related model, known as $r$-neighbour bootstrap percolation, which was introduced in 1979 by Chalupa, Leath and Reich~\cite{CLR} as a model of disordered magnetic systems. In this process, a vertex of a graph $G$ becomes infected when it has at least $r$ infected neighbours, and infected vertices remain infected forever. We remark that this is similar to $H$-bootstrap percolation with $H$ a star, except that a given copy of $H$ can only be responsible for infecting its central vertex; it is thus a special case of the natural `directed' version of $\HH$-bootstrap percolation, in which each edge can infect only a single specified vertex. The $r$-neighbour bootstrap process has been extensively studied by mathematicians and statistical physicists (see~\cite{AL,BBDM,Hol}, for example, and the references therein). For further background see Bollob\'as~\cite{bb11}.

In $r$-neighbour bootstrap percolation, one is mainly interested in estimating the 
\emph{critical threshold} in the random setting: if the initially infected set $A$ is formed by
selecting vertices independently with probability $p$, for which $p$ is
it likely that eventually all vertices are infected?  In the study of
this probabilistic question, extremal results turn out to be important
(see~\cite{BBMhigh} or~\cite{GHM}, for example). One of our main
motivations in this work is to approach the following tantalizing open
problem, which is our main stumbling block in attacking the
probabilistic question on the hypercube. Let $m(G,r)$ denote the
minimum size of a percolating set in $r$-neighbour bootstrap
percolation on $G$. In \cite{BB}, Balogh and Bollob\'as made the
following conjecture.

\begin{conj}\label{BBconj}
 Let $r \ge 3$ be fixed. Then
 $$m(Q_d,r) \; = \; \left( \frac{1}{r} + o(1) \right) \binom{d}{r-1}$$
 as $d \to \infty$.
\end{conj}

The upper bound in Conjecture~\ref{BBconj} follows by taking a Steiner system at level $r$, together with all of level $r-2$. Amazingly, we know of no super-linear lower bound. In the case $r = 2$ the situation is simpler, and $m(P_n^d,2)$ is known exactly for all $n$ and $d$ (see~\cite{BB} or~\cite{BBMhigh}). At the other end of the range, Pete (see~\cite{BP}) observed that $m(P_n^d,d) = n^{d-1}$.  However, for fixed $2 < r < d$, $m(P_n^d,r)$ is known only up to a constant factor that depends on $d$. 

Finally, we remark that the random questions are also interesting in the $H$-bootstrap model, and that some of the basic problems (in the `edge version') are solved in~\cite{BBMgb} by the first three authors. As the reader might guess, however, there are still many more open problems than theorems.

The rest of this note is arranged as follows.
In Section~\ref{oliver} we prove Theorem~\ref{mainthm}, and in Section~\ref{discuss} we discuss
an inhomogeneous extension.

\section{Proof of Theorem~\ref{mainthm}}\label{oliver}

The proof of Theorem~\ref{mainthm} is based on the following observation.

\begin{lemma}\label{ldep}
Let $\HH$ be an arbitrary hypergraph. Suppose that we can find a vector space $W$
spanned by vectors $\{ f_v : v \in V(\HH)\}$ such that, for every edge $S \in \HH$,
we have a linear dependence $\sum_{v \in S} \lambda_{S,v} f_v=0$ with all coefficients $\lambda_{S,v}$ non-zero.
Then
\[
 m(\HH) \ge \dim W.
\]
\end{lemma}

\begin{proof}
Once one thinks of the statement, the proof is essentially immediate. Indeed, suppose that $A \subset V(\HH)$ percolates in the $\HH$-process. Then we can order the vertices $v_1,\ldots,v_\ell$ in $V(\HH) \setminus A$ so that each $v_i$ is in an edge $S_i$ of $\HH$
in $A_i = A \cup \{ v_j : j\le i \}$. Let $W_i$ be the span of the vectors $\{f_v : v\in A_i\}$.
For each $i\ge 1$, the dependency condition for $S_i$ and that fact all other vertices of $S_i$ are in $A_{i-1}$
together imply that $f_{v_i}$ is a linear combination of vectors in $W_{i-1}$. Indeed, there exist $\lambda_{S_i,v}$, one for each $v \in S_i$ and all non-zero, such that
$$f_{v_i} \, = \, - \frac{1}{\lambda_{S_i,v_i}} \sum_{v \in S_i \setminus \{v_i\}} \lambda_{S_i,v} f_v.$$
Thus $W_i = W_{i-1}$ and hence $W_\ell = W_0$. By assumption, $A_\ell = V(\HH)$, so $W_\ell = W$. Since $W_0$ is spanned by $|A|$ vectors, we have $|A| \ge \dim W$.
\end{proof}

To prove Theorem~\ref{mainthm}, we must find the right vectors. Since the notation in the formal proof
may perhaps obscure the ideas, we first outline some special cases. Given $v \in [n]^d$ we write
$|v|$ for the sum of the coordinates of $v$.

Call a coordinate $i$ of a vector $v$ \emph{large} if $v_i \ge t$ and \emph{small} otherwise, and let
$$U \, = \, \big\{ (u_1,\ldots,u_d)\in [n]^d \,:\, \big| \big\{ i : u_i \ge t \big\} \big| \le r-1 \big\}$$
be the set of all $v \in [n]^d$ with at most $r-1$ large coordinates, as in~\eqref{Udef}. We begin by assigning to each vertex $u \in U$ an abstract vector $e_u$, and assume that the vectors $\{e_u : u\in U\}$ are linearly independent. 

Suppose first that $t = 2$ and $r = d$, so the edges $S$ of our hypergraph $\HH$ are hypercubes $Q_d$ of full dimension in $[n]^d$, and $U = \{\x \in [n]^d : \min_j x_j = 1\}$ is a union of $(d-1)$-dimensional faces of $[n]^d$. Now, for each $v \in [n]^d$ and $k \in [d]$, let $\pi_k(v)$ denote the projection of $v$ onto the face $\{\x \in [n]^d : x_k = 1\}$, and set 
$$f_v \, = \, \sum_{k=1}^d e_{\pi_k(v)}.$$
We claim that the dependency condition in Lemma~\ref{ldep} holds with $\lambda_{S,v} = \pm 1$ for each $S$ and $v$, simply by letting the sign alternate over the vertices of the cube $S$ in the obvious way. To see this, consider the contribution to 
$$\sum_{v \in S} \lambda_{S,v} f_v=\sum_{v\in S} \lambda_{S,v} \sum_{k=1}^d e_{\pi_k(v)} \, = \, \sum_{k=1}^d \sum_{v\in S} \lambda_{S,v} e_{\pi_k(v)}$$ 
from terms with a given value of $k$. The vertices of $S$ can be grouped into pairs $\{v,v'\}$ differing only in the $k^{th}$ coordinate, and so with $\pi_k(v) = \pi_k(v')$. The choice
of sign ensures that the corresponding contributions to the sum cancel. One must also check that the vectors $\{ f_v : v \in [n]^d \}$ 
have the same span as the vectors $\{ e_u : u \in U\}$; this follows from the fact that, for each $v \in U$, $f_v$ is equal to $c_v e_v$ plus a sum of terms involving $e_u$ with $|u| < |v|$, where $c_v > 0$ is the number of small coordinates of $v$.

The case $t = 2$ and general $1 \le r \le d$ is not much harder: now $U$ consists of the $(r-1)$-dimensional faces $\{\x \in [n]^d : x_{j_1} = \ldots = x_{j_{d-r+1}} = 1\}$, and we define $\pi_T(v)$ to be the projection onto the face indexed by $T = \{j_1,\ldots,j_{d-r+1} \}$, and set
$$f_v \, = \, \sum_{T} e_{\pi_T(v)},$$
where the sum runs over all sets $T \subset [d]$ of size $d - r + 1$.  Setting $\lambda_{S,v} = \pm 1$ for each $S$ and $v$, exactly as before, we can again group the vertices of $S$ into pairs that project to the same point; the proof is now exactly as above.

Now suppose that $r = d$, so we are back to projecting out single coordinates, but that $t \ge 3$. In order to define the vectors $f_v$, consider first a single coordinate $k$, and a line $v(1),\ldots,v(n)$ of points of $[n]^d$ differing only in the $k^{th}$ coordinate, with $v(i)$ having $k^{th}$ coordinate $i$. We shall set
$$f_v = \sum_{k=1}^d f^{(k)}_v,$$
where each $f^{(k)}_{v(i)}$ is a linear combination of the vectors $e_{v(1)},\ldots,e_{v(t-1)}$. For those $v(i)$ with $k^{th}$ coordinate small (i.e., $i \le t-1$), we just take $f^{(k)}_{v(i)} = e_{v(i)}$; for those with $k^{th}$ coordinate large, we want to choose the linear combinations such that the following holds: 
\begin{equation}\label{keyprop}
\textup{For any $i_1 < \ldots < i_t$, the vectors $f^{(k)}_{v(i_j)}$ are `minimally' dependent,}
\end{equation}
meaning that they satisfy a linear dependence with all coefficients non-zero.
We remark that the vectors $v(i_1), \ldots, v(i_t)$ will be $t$ points in a line in an edge of $\HH$, i.e., a copy of $K_t^r$. It is clearly possible to choose the linear combinations so that~\eqref{keyprop} holds, simply by choosing the linear combinations to be `in general position'. Now, given an edge $S$ of $\HH$, we take the coefficients of these dependencies (which we take to be the same for all lines in direction $k$) as one factor contributing to $\lambda_{S,v}$; there is a similar factor for each coordinate $k$. (Note that this is in fact exactly what we did in the case $t = 2$: there, for a given $k$, each $f^{(k)}_{v(i)}$ is simply $e_{v(1)}$, and the coefficients of our dependency are $\pm 1$.) The proof now follows as before.

In the fully general case we sum over all projections onto `thickened' $(r-1)$-dimensional faces. When
projecting out $d - r + 1 \ge 2$ coordinates, we choose coefficients for the $(t-1)^{d-r+1}$ allowed image
vectors by multiplying the coefficients associated to projecting out a single coordinate; a formal
description follows.

\begin{proof}[Proof of Theorem \ref{mainthm}]
Fix $n\ge t\ge 2$ and $d \ge r \ge 1$, and set $\K = \K(n,d,t,r)$ and $\P = \P(n,d,t,r)$.
As above, for $v = (v_1,\ldots,v_d) \in [n]^d$ we say that coordinate $i$ of $v$ is \emph{large} if $v_i \ge t$ and \emph{small} otherwise. As in the statement of the theorem, let $U$ be the set of all $v\in [n]^d$ with at most $r-1$ large coordinates, and note that 
$$|U| \, = \, \sum_{s=0}^{r-1}\binom{d}{s} (t-1)^{d-s} (n-t+1)^s.$$
Suppose that the set of initially infected vertices is exactly $U$. Then every vertex $v$ is eventually infected in the $\P$-process. Indeed, for $v\notin U$ we can use any $r$ large coordinates of $v$ to construct a copy $H$ of $P_t^r$ in $P_n^d$ with $v$ as the `top' vertex, i.e., with $|u| < |v|$ for all other vertices $u$ of $H$. It follows by induction on $|v|$ that all $v$ are infected eventually. Hence, since $\P \subset \K$, 
\begin{equation}\label{ub1}
m(\K) \,\le\, m(\P) \, \le \, |U|.
\end{equation}

For the lower bound let $W$ be a (real) vector space with basis $\{e_u : u \in U\}$, so $\dim W = |U|$. (Here the $e_u$ are simply abstract linearly independent vectors.)
By Lemma~\ref{ldep} it suffices to define vectors $\{f_v:v\in [n]^d\}$ with $f_v\in W$ in such a way that $(i)$ $\{f_v\}$ spans $W$ and $(ii)$ for every $S\subset [n]^d$ which is an edge of $\K(n,d,t,r)$ 
there are non-zero coefficients $\lambda_{S,v}\in \RR$ such that
\begin{equation}\label{dep}
 \sum_{v\in S} \lambda_{S,v}f_v =0.
\end{equation}

Let us fix once and for all an $n \times (t-1)$  matrix $M = (m_{ij})_{1\le i\le n, \,1\le j\le t-1}$ with the following properties:
the first $t-1$ rows of $M$ form the identity matrix, all entries of $M$ are non-negative,
and any $t-1$ rows of $M$ are linearly independent. Such matrices clearly exist: in constructing the next
row (after the first $t-1$) we just avoid the union of a finite number of $(t-2)$-dimensional subspaces.
Since any $t$ rows of $M$ are dependent, but no $t-1$ rows are, for any subset $I \subset [n]$ with $|I|=t$ there are \emph{non-zero} coefficients $(\lambda_{I,i})_{i\in I}$ such that 
\begin{equation}\label{Idef}
 \sum_{i\in I} \lambda_{I,i} m_{ij} =0 \hbox{\quad for each \quad} j=1,2,\ldots,t-1.
\end{equation}

Given a point $v=(v_1,\ldots,v_d) \in [n]^d$, a coordinate $1\le k\le d$ and a `value' $j\in [n]$, let $\pi^k_j(v)$ denote
the point $(v_1,\ldots,v_{k-1},j,v_{k+1},\ldots,v_d)$ obtained by starting from $v$ and setting the value
of the $k^{th}$ coordinate to $j$. Similarly, for \emph{distinct} $k_1,k_2,\ldots,k_p\in [d]$ and arbitrary
$j_1,\ldots,j_p\in [n]$ let 
$$\pi^{k_1,\ldots,k_p}_{j_1,\ldots,j_p}(v) \, = \, \pi^{k_1}_{j_1}( \dots \pi^{k_p}_{j_p} (v)\dots)$$ 
be the point obtained from $v$ by setting
the value of coordinate $k_i$ to $j_i$ for $1\le i\le p$.

Let $P=\{k_1,\ldots,k_p\}$ be any subset of $[d]$ with cardinality $p=d-r+1$; these will be
the coordinates `projected out'.
For $v\in [n]^d$, define
\begin{equation}\label{fvp}
 f^{(P)}_{v} = \sum_{j_1=1}^{t-1}\cdots\sum_{j_p=1}^{t-1} \prod_{a=1}^p m_{i_a j_a} e_{ \pi^{k_1,\ldots,k_p}_{j_1,\ldots,j_p}(v) },
\end{equation}
where $i_a=v_{k_a}$ is the original value of the $a^{th}$ coordinate which is projected out; we shall see in a moment
that this defines a vector in $W$.  We set
\begin{equation}\label{fv}
 f_v = \sum_P f^{(P)}_{v},
\end{equation}
where the sum is over all $P\subset [d]$ with $|P|=p$.

We shall break up the remaining calculation into three simple claims.

\begin{claim}
$f_v\in W$ for each $v \in [n]^d$.
\end{claim} 

\begin{proof}
Let $v \in [n]^d$ and $P \subset [d]$ with $|P| = d - r + 1$, and suppose that $e_u$ appears in the sum in \eqref{fvp}, with $u = \pi^{k_1,\ldots,k_p}_{j_1,\ldots,j_p}(v)$. Then the coordinates which are projected out (the $k_a$) take values (the $j_a$) that are small. Since only $d-p=r-1$ coordinates are not
projected out, $u$ has at most $r-1$ large coordinates, so $u\in U$, and $e_u\in W$. Thus $f^{(P)}_{v}\in W$ and hence $f_v\in W$.
\end{proof}

\begin{claim}
The vectors $\{f_v: v\in [n]^d\}$ span $W$.
\end{claim}

\begin{proof}
Recall that $|v|$ denotes the sum of the coordinates of $v$. We shall show that for every $v \in U$, there exist constants $c_v > 0$ and $\mu_{uv} \in \RR$ such that
\begin{equation}\label{eq:C2}
 f_v \, = \, c_v e_v + \sum_{u \in U \,:\, |u| < |v|} \mu_{uv} e_u.
\end{equation}
To prove~\eqref{eq:C2}, note first that $m_{ij} \ne 0$ implies that $j \le i$, since the first $t-1$ rows of $M$ form the identity matrix, and so either $i = j \le t-1$, or $i \ge t$ and $j \le t-1$. By~\eqref{fvp}, it follows that $f^{(P)}_{v}$ (and hence $f_v$) is a linear combination of vectors $e_u$ with $|u| \le |v|$, for every $v \in [n]^d$ and every $P \subset [d]$ with $|P| = d - r + 1$. Moreover, since $u\le v$ holds coordinate-wise, we have $|u| < |v|$ unless $u = v$. 

Now, suppose that $v \in U$. Then $v$ has at least $p$
small coordinates, so we can choose some $P$ consisting only of small coordinates
of $v$. Taking $j_a=i_a$ in each sum in~\eqref{fvp}, we see that $e_v$ appears at least
once in the formula defining $f^{(P)}_{v}$. Since we chose all of the $m_{ij}$ to be non-negative, it follows that all our coefficients are non-negative. Thus $f^{(P)}_{v}$, and hence $f_v$, has a strictly positive coefficient of $e_v$, as required.

From~\eqref{eq:C2} it follows by elementary linear algebra that the vectors $\{ f_v : v \in U\}$ span $W$.
Indeed, writing $\spn{\cdot}$ for the linear span of a set of vectors, one can show by induction on $s$ 
that for each $s$, 
$$\spn{\{f_v: v\in U, |v|\le s\}}=   \spn{\{e_v: v\in U: |v|\le s\}}.$$ 
Hence the vectors $\{f_v: v\in [n]^d\}$ span $W$, as claimed.
\end{proof}

It remains only to establish the dependency condition~\eqref{dep}; the first step is to define the coefficients $\lambda_{S,v}$. For each edge $S$ of $\K(n,d,t,r)$, i.e., for each induced copy of $K_t^r$ in $[n]^d$, let $D(S)$ denote the set of coordinates in which $S$ varies, so $|D| = r$. Moreover, for each $\ell \in D(S)$ let $I_\ell = I_\ell(S)$ denote the set of values taken by the $\ell$-coordinates of points
in $S$, so $|I_\ell| = t$. Now, for each $v=(v_1,\ldots,v_d) \in S$, set
\begin{equation}\label{ladef}
 \lambda_{S,v} = \prod_{\ell \in D(S)} \lambda_{I_\ell,v_\ell},
\end{equation}
where the coefficients $\lambda_{I,i}$ are as in~\eqref{Idef}. Note that $\lambda_{S,v} \ne 0$ for all $v \in S$.

\begin{claim}
For each $S \in \K$ and each $P \subset [d]$ with $|P| = d - r + 1$, we have
\begin{equation}\label{depP}
 \sum_{v \in S} \lambda_{S,v} f^{(P)}_{v} = 0.
\end{equation}
\end{claim}

\begin{proof}
We shall establish~\eqref{depP} by partitioning the vertices of $S$ into lines, and showing that the sum over the $t$ vertices $v$ in any such line $L$ is zero. To do so, recall that $|P| = d - r + 1$ and $|D(S)| = r$, so $|P \cap D(S)| \ge 1$. Let $k \in P \cap D(S)$, and partition the $t^r$ vertices of $S$ into $t^{r-1}$ lines of size $t$, i.e., sets of vertices in $S$ differing only in the $k^{th}$ coordinate. 

Let $L$ be one of these lines, and observe that if $w \in L$ then $L = \{ \pi^k_i(w) : i \in I_k\}$. By~\eqref{Idef}, for every $1 \le j \le t-1$ we have $\sum_{i\in I_k} \lambda_{I_k,i} m_{i j}=0$, and hence
\begin{equation}\label{eq:C3}
\sum_{j=1}^{t-1} \sum_{i\in I_k} \lambda_{I_k,i} m_{i j} e_{ \pi^{k_1,\ldots,k_{p-1},k}_{j_1,\ldots,j_{p-1},j}(w) } \, = \, 0
\end{equation}
for every $1 \le j_1,\ldots,j_{p-1} \le t-1$ and every $k_1,\ldots,k_{p-1} \ne k$. It follows by~\eqref{fvp} that
$$\sum_{i\in I_k} \lambda_{I_k,i} f^{(P)}_{\pi^k_i(w)} \, = \, \sum_{j_1=1}^{t-1}\cdots\sum_{j_{p-1}=1}^{t-1} \bigg( \prod_{a=1}^{p-1} m_{i_a j_a} \bigg) \sum_{j=1}^{t-1} \sum_{i\in I_k} \lambda_{I_k,i} m_{i j} e_{ \pi^{k_1,\ldots,k_{p-1},k}_{j_1,\ldots,j_{p-1},j}(w) } \, = \, 0,$$
where $P = \{k_1,\ldots,k_p\}$ with $k_p = k$, and (for each $a < p$) $i_a$ is the $k_a^{th}$ coordinate of $w$, and hence of $\pi^k_i(w)$, since $k_a \ne k$.

Finally observe that, in the product~\eqref{ladef}, as $v$ ranges over the line $L$ only the factor corresponding to coordinate $k$ varies. Thus
\begin{equation}\label{end}
 \sum_{v\in L} \lambda_{S,v} f^{(P)}_{v} \, = \, \sum_{i\in I_k} \lambda_{S,\pi^k_i(w)} f^{(P)}_{\pi^k_i(w)} 
\, = \, \bigg( \prod_{\ell \in D\setminus\{k\}} \lambda_{I_{\ell},w_{\ell}} \bigg) \sum_{i\in I_k} \lambda_{I_k,i} f^{(P)}_{\pi^k_i(w)} \, = \, 0.
\end{equation}
Summing~\eqref{end} over all $r^{t-1}$ lines $L$ proves the claim. 
\end{proof}

Finally, observe that~\eqref{dep} follows immediately from Claim~3 by summing over $P$. Combining~\eqref{dep} with Claims~1 and~2, and applying Lemma~\ref{ldep}, it follows that 
$$m(\K) \, \ge \, |U| \, = \, \sum_{s=0}^{r-1}\binom{d}{s} (t-1)^{d-s} (n-t+1)^s,$$
as required. This completes the proof of Theorem~\ref{mainthm}.
\end{proof}

\section{An inhomogeneous extension}\label{discuss}

In Theorem~\ref{mainthm} we took the dimensions (that is, the side-lengths) of our `host graph' $K_n^d$ or $P_n^d$ to be all equal to $n$ purely for notational convenience. The proof is unaltered if we replace $[n]^d$ by $[n_1] \times \cdots \times [n_d]$;
we simply take a matrix $M$ with at least $\max_k n_k$ rows, and define $f^{(P)}_{v}$ exactly as before. 

In fact, there is no need to use the same matrix for different coordinate axes: the proof still works
if we use different matrices $M^{(k)}$ for different coordinates $1\le k\le d$. This allows us to consider
graphs $H$ that are grids with different dimensions in different coordinates. However, when $r<d$,
we must be careful -- we must fix numbers $t_1,\ldots,t_d$ and consider
$r$-dimensional grids $H$ with the property that if coordinate $k$ varies in $H$ then
it takes $t_k$ distinct values. If this seems unnatural, then the reader should consider only the case $r = d$, which corresponds to the result of Alon~\cite{alon} mentioned earlier.

Define $\K \, = \, \K(n_1,\ldots,n_d,t_1,\ldots,t_d,r)$ to be the hypergraph with vertex set $[n_1] \times \cdots \times [n_d]$ in which the hyperedges are all sets $S$
of the form $S = I_1 \times \cdots \times I_d$ with $I_k \subset [n_k]$ for each $k$, and either $|I_k| = 1$ or $|I_k| = t_k$, with the second case occurring for exactly $r$ coordinates $k$. Define $\P\subset \K$ similarly, but with the extra restriction
that each $I_k$ is an interval. Finally, let $U$ be the set of vertices $v = (v_1,\ldots,v_d)$
in which at most $r-1$ coordinates $v_k$ satisfy $v_k\ge t_k$. The following theorem says that the set $U$ is extremal in both $\P$ and $\K$.

\begin{thm}\label{genthm}
Let $d \ge r \ge 1$ and let $n_k \ge t_k \ge 2$ for each $1\le k \le d$. Then 
$$m(\K) \, = \, m(\P) \, = \sum_{S \subset [d] \,:\, |S| \le r-1} 
\bigg( \prod_{k \in S} \big(n_k + 1 - t_k\big) \bigg)
 \bigg( \prod_{k \in S^{\mathrm{c}}} \big( t_k - 1 \big) \bigg) .$$
\end{thm}

\begin{proof}
The theorem follows by a simple adaptation of the proof of Theorem~\ref{mainthm} above, taking the $k^{th}$ coordinate
of $v = (v_1,\ldots,v_d)$ to be large if $v_k \ge t_k$.  In place of the matrix $M$, we choose an $n_k \times (t_k-1)$
matrix $M^{(k)}$ for each coordinate $k$, and replace~\eqref{fvp} by
\[
 f^{(P)}_{v} = \sum_{j_1=1}^{t_{k_1}-1}\cdots\sum_{j_p=1}^{t_{k_p}-1} \prod_{a=1}^p m^{(k_a)}_{i_a j_a} e_{ \pi^{k_1,\ldots,k_p}_{j_1,\ldots,j_p}(v) }.
\]
The remainder of the proof is identical.
\end{proof}

\section*{Acknowledgements}

The authors would like to thank an anonymous referee for pointing out that Theorem~3.4 from~\cite{alon} implies Corollary~\ref{aloncor}.

\end{document}